\documentclass [12pt]{article}

\usepackage [latin2] {inputenc}
\usepackage [british]{babel}
\usepackage {t1enc}
\usepackage {calc}
\usepackage {latexsym}
\usepackage {amssymb}
\usepackage {amsmath}
\usepackage [pdftex]{graphicx}
\usepackage {wrapfig}
\usepackage {bbm}
\usepackage {verbatim}
\usepackage {hyperref}
\AtBeginDocument{
  \let\oldref\ref 
  \def\ref{\oldref*}}
\newtheorem{theorem}{Theorem}[section]
\newtheorem{lemma}[theorem]{Lemma}
\newtheorem{proposition}[theorem]{Proposition}
\newtheorem{corollary}[theorem]{Corollary}

\newenvironment{proof}[1][Proof]{\begin{trivlist}
\item[\hskip \labelsep {\bfseries #1}]}{\hfill $\square$\end{trivlist}}
\newenvironment{definition}[1][Definition]{\begin{trivlist}
\item[\hskip \labelsep {\bfseries #1}]}{\end{trivlist}}

\newenvironment{remark}[1][Remark]{\begin{trivlist}
\item[\hskip \labelsep {\bfseries #1}]}{\end{trivlist}}
\newenvironment{remarks}[1][Remarks]{\begin{trivlist}
\item[\hskip \labelsep {\bfseries #1}]}{\end{trivlist}}
\newenvironment{theorem*}[1][Theorem]{\begin{trivlist}
\item[\hskip \labelsep {\bfseries #1}]}{\end{trivlist}}

\begin{document}
\title{On the Schneider-Vigneras functor \\ for principal series}
\author{M\'arton Erd\'elyi \\ Alfr\'ed R\'enyi Institute of Mathematics, Budapest \\ \texttt{merdelyi@freestart.hu}}

\maketitle

\begin{abstract}
We study the Schneider-Vigneras functor attaching a module over the Iwasawa algebra $\Lambda(N_0)$ to a $B$-representation for irreducible modulo $\pi$ principal series of the group $\mathrm{GL}_n(F)$ for any finite field extension $F|\mathbb{Q}_p$.

\end{abstract}

\section{Introduction}

Let $\mathbb{Q}_p$ be the field of $p$-adic numbers, $\overline{\mathbb{Q}}_p$ its algebraic closure, $F,K\leq\overline{\mathbb{Q}}_p$ finite extensions of $\mathbb{Q}_p$. Let $o_F$, respectively $o_K$ be the rings of integers in $F$, respectively in $K$, $\pi_F\in o_F$ and $\pi_K\in o_K$ uniformizers, $\nu_F$ and $\nu_K$ the standard valuations and $k_F=o_F/\pi_Fo_F$, $k_K=o_K/\pi_Ko_K$ the residue fields.

The Langlands philosophy predicts a natural correspondence between certain admissible unitary representations of $\mathrm{GL}_n(F)$ over Banach $K$-vector spaces and certain $n$-dimensional $K$-representations of the Galois-group \break $\mathrm{Gal}(\overline{\mathbb{Q}}_p|F)$.

Colmez proved the existence of such a correspondence in the case of $\mathrm{GL}_2(\mathbb{Q}_p)$, but for any other group even the conjectural picture is not developed yet. It turned out, that Fontaine's theory of $(\varphi,\Gamma)$-modules is a fundamental intermediary between the representations of $\mathrm{Gal}(\overline{\mathbb{Q}_p}/\mathbb{Q}_p)$ and the representations of $\mathrm{GL}_2(\mathbb{Q}_p)$. Schneider and Vigneras managed to generalize parts of Colmez's work to reductive groups other than $\mathrm{GL}_2(\mathbb{Q}_p)$.

Our aim is to understand the construction of Schneider and Vigneras, attaching a generalized $(\varphi, \Gamma)$-module to a smooth torsion $o_K$-representation of $G$, for principal series representations $V$ in the case $G=\mathrm{GL}_n(F)$. Originally this functor (which we denote by $D$) is defined only for $F=\mathbb{Q}_p$, but our considerations work for any finite extension $F|\mathbb{Q}_p$ and the analogous definitions.

~\\

In order to that, we need to understand the $B_+$-module structure of the principal series, where $B_+$ is a certain submodule of a Borel subgroup $B$ in $G$. In section 3 we decompose $G$ to open $N_0$-invariant subsets (where $N_0$ is a totally decomposed compact open subgroup in the unipotent radical of $B$), indexed by the Weyl group.

With the help of this in section 4 we prove that there exists a minimal element $M_0$ in the set of generating $B_+$-subrepresentations of $V$.

Now we have that $D(V)=M_0^*$ - the dual of this minimal $B_+$-sub\-rep\-resen\-tation. We do not know whether it is finitely generated or it has rank 1 as a module over $\Omega(N_0)=\Lambda(N_0)/\pi_K\Lambda(N_0)$ (where $\Lambda(N_0)$ is the Iwasawa algebra of $N_0$). However, we show that in some sense only a rank 1 quotient of $D(V)$ is relevant if we want to get an \'etale $(\varphi,\Gamma)$-module.

In the last section we point out some properties of $M_0$, which sheds some light on why the picture is more difficult for principal series than in the case of subquotients defined by the Bruhat filtration.

~\\

\textbf{Acknowledgments.} I gratefully acknowledge the financial support and hospitality of the Central European University, and the Alfr\'ed R\'enyi Institute of Mathematics, both in Budapest.
I would like to thank Gergely Z\'abr\'adi for introducing me to this field, and for his constant help, valuable comments and all the useful discussions. I am grateful to the anymous referee for the very careful reading of the paper and the helpful remarks, and also to Levente Nagy for reading through an earlier version.

\newpage

\section{Notations}

Let $G$ be the $F$-points of a $F$-split connected reductive group over $\mathbb{Q}_p$. Let $B\leq G$ be a fixed Borel subgroup, with maximal torus $T$ and unipotent radical $N$. Let $W\simeq N_G(T)/C_G(T)$ be the Weyl group of $G$, $\Phi^+$ the set of positive roots with respect to $B$, and $N_\alpha$ denote the root subgroup for each $\alpha\in\Phi^+$. A subgroup $N_0\leq N$ is called totally decomposed if for any total ordering of $\Phi^+$ we have $N_0=\prod_{\alpha\in\Phi^+}(N_0\cap N_\alpha)$.

As an $o_K$-representation of $G$ we mean a pair $V=(V,\rho)$, where $V$ is a torsion $o_K$-module, $\rho:G\to\mathrm{GL}(V)$ is a group homomorphism. $V$ is smooth if $\rho$ is locally constant ($\forall v\in V~\exists U\subset G$ open, such that $\forall u\in U: \rho(u)v=v$). $V$ is admissible if for any $U\leq G$ open subgroup, the vector space $k_K\otimes_{o_K}V^U$ is finite dimensional.

For an $o_K$-representation $V$ let $V^*=\mathrm{Hom}_{o_K}(V,K/o_K)$ be the Pontrjagin dual of $V$. Pontrjagin duality sets up an anti-equivalence between the category of torsion $o_K$-modules and the category of all compact linear-topological $o_K$-modules.

Let $G_0\leq G$ be a compact open subgroup and $\Lambda(G_0)$ denote the completed group ring of the profinite group $G_0$ over $o_K$. Any smooth $o_K$-representation $V$ is the union of its finite $G_0$-subrepresentations, therefore $V^*$ is a left $\Lambda(G_0)$-module (through the inversion map on $G_0$).

Let $\Omega(G_0)=\Lambda(G_0)/\pi_K\Lambda(G_0)$. $\Omega(N_0)$ is noetherian and has no zero divisors, so it has a fraction (skew) field. If $M$ is a $\Omega(N_0)$-module, by the rank of $M$ we mean $\dim_{k_K}(\mathrm{Frac}(\Omega(N_0))\otimes_{\Omega(N_0)}M)$.

~\\

From now on fix $n\in\mathbb{N}$, and let $G=\mathrm{GL}_n(F)$, and $G_0=\mathrm{GL}_n(o_F)$.

Let $B$ be the set of upper triangular matrices in $G$, $T$ the set of diagonal matrices, $N$ the set of upper triangular unipotent matrices. Let $N^-$ be the lower unipotent matrices - the opposite of $N$ - and $N_0=N\cap G_0$ - a totally decomposed compact open subgroup of $N$ - those matrices wich has coefficients in $o_F$, define the following submonoid of $T$:
\[T_+=\{t\in T|tN_0t^{-1}\subset N_0\}=\{\mathrm{diag}(x_1,x_2,\dots,x_n)|i>j:\nu_F(x_i)\geq\nu_F(x_j)\}.\]

We have the following partial ordering on $T_+$: $t\leq t'$ if there exists $t''\in T_+$ such that $tt''=t'$. Let $B_+=N_0T_+$, this is a submonoid of $B$.

By the abuse of notation let $w\in W$ denote also the permutation matrices - representatives of $W$ in $G$ (with $w_{ij}=1$ if $w(j)=i$, and $w_{ij}=0$ otherwise), and also the corresponding permutation of the set $\{1,2,\dots,n\}$. For $w\in W$ denote length of $w$ - the length of the shortest word representing $w$ in the terms of the standard generators of $W$ - by $l(w)$.

Let the kernel of the projection $pr:G_0\to\mathrm{GL}_n(k_F)$ be $U^{(1)}$. This is a compact open pro-$p$ normal subgroup of $G_0$. We have $G=G_0B$ and $U^{(1)}\subset(N^-\cap U^{(1)})B$.

~\\

Let $C^\infty(G)$ (respectively $C^\infty_c(G)$) denote the set of locally constant \break $G\to k_K$ functions (respectively locally constant functions with compact support), with the group $G$ acting by left multiplication ($gf:x\mapsto f(g^{-1}x)$ for $f\in C^\infty(G)$ and $g,x\in G$). Let
\[\chi=\chi_1\otimes\chi_2\otimes\dots\otimes\chi_n:T\to k_K^*\]
be a locally constant character of $T$ with $\chi_i:F^*\to k_K^*$ multiplicative. Note that then for all $i$ $\chi_i(1+\pi_Fo_F)=1$ and $\chi_i(o_F^*)\subset k_F^*\cap k_K^*\leq\overline{\mathbb{F}_p}^*$. Since $T\simeq B/[B,B]$, also denote the correspondig $B\to k_K^*$ character by $\chi$. Let 
\[V=\mathrm{Ind}^{G}_{B}(\chi)=\{f\in C^\infty(G)|\forall g\in G, b\in B:f(gb)=\chi^{-1}(b)f(g)\}\]
$V$ is called a principal series representation of $G$. $V$ is irreducible exactly when for all $i$ we have $\chi_i\neq\chi_{i+1}$ (\cite{Ol}, theorem 4).
For any open right $B$-invariant subset $X\subset G$ we write $\mathrm{Ind}^X_B=\{F\in\mathrm{Ind}^G_B(\chi)|F|_{G\setminus X}\equiv0\}$.

We can understand the stucture of $V$ better (see \cite{Vi}, section 4.), by the Bruhat decomposition $G=\bigcup_{w\in W}BwB$. Let $\prec$ denote the strong Bruhat ordering (see \cite{Jan} II. 13.7): we say $w'\prec w$ for $w\neq w'\in W$ if there exist transpositions $w_1,w_2,\dots, w_i\in W$ such that $w'=ww_1w_2\dots w_i$ and \break $l(w)>l(ww_1)>l(ww_1w_2)>\dots>l(ww_1w_2\dots w_i)$. Fix a total ordering $\prec_T$ refining the Bruhat ordering $\prec$ of $W$, and let
\[w_1=\mathrm{id}_W\prec_T w_2\prec_T w_3\prec_T\dots\prec_T w_{n!}=w_0.\]
Let us denote by $G_m=\bigcup_{1\leq l\leq m}Bw_lB$ - a closed subset of $G$. We obtain a descending $B$-invariant filtration of $V$ by
\[V_m=\mathrm{Ind}^{G\setminus G_m}_B(\chi)=\{F\in\mathrm{Ind}^G_B(\chi)|F|_{G_m}\equiv0\}\qquad(0< m\leq n!),\]
with quotients $V_{m-1}/V_m$ via $f\mapsto f(\cdot w_m)$ isomorphic to \break $V(w_m,\chi)=C^\infty_c(N/N'_{w_m})$ (see \cite{SchVi}, section 12), where $N'_{w_m}=N\cap w_mNw_m^{-1}$, with $N$ acting by left translations and $T$ acting via
\[(t\phi)(n)=\chi(w_m^{-1}tw_m)\phi(t^{-1}nt).\]

For any $w\in W$ put
\[N_w=\{n\in N|\forall i<j, w^{-1}(i)<w^{-1}(j):n_{ij}=0\}=N\cap wN^-w^{-1}\leq N,\]
and $N_{0,w}=N_0\cap N_w$. Then we have the following form of the Bruhat decomposition $G=\coprod_{w\in W}N_wwB$.

\section{The action of $B_+$ on $G$}

The first goal is to partition $G$ to $N_0$-invariant open subsets $\{U_w|w\in W\}$ indexed by the Weyl-group, which are respected by the $B_+$-action in the sense that if $x\in U_w$ $b\in B_+$ then there exists $w'\preceq w$ in $W$ such that $b^{-1}x\in U_{w'}$.

\begin{definition}
Let for any $w\in W$ $r_w:N^-\cap G_0\to G(k_F), n^-\mapsto pr(wn^-w^{-1})$, $R_w=wr_w^{-1}(N_0(k_F))$, $R=\cup_{w\in W}R_w$.
\end{definition}

We have that
\[R_w=\left\{(a_{ij})\in G|\forall i,j: a_{ij}\left\{\begin{array}{ll}=1, & \mathrm{if~}w^{-1}(i)=j \\ =0, & \mathrm{if~}w^{-1}(i)<j \\ \in o_F, & \mathrm{if~}w^{-1}(i)>j\mathrm{~and~}w(j)>i \\ \in \pi_F o_F, & \mathrm{if~}w^{-1}(i)>j\mathrm{~and~}w(j)<i\end{array}\right.\right\}\]

For $n=3$ in details (with $o=o_F$ and $\pi=\pi_F$):

~\\
\noindent
\bgroup
\def\arraystretch{0.8}
\setlength\tabcolsep{1pt}
\begin{tabular}{|c|c||c|c|}
\hline
$w$ & $R_w$ & $w$ & $R_w$ \\
\hline
$\mathrm{id}=\left(\begin{array}{ccc} 1 & 0 & 0 \\ 0 & 1 & 0 \\ 0 & 0 & 1\end{array}\right)$ & $\left(\begin{array}{ccc} 1 & 0 & 0 \\ \pi o & 1 & 0 \\ \pi o & \pi o & 1\end{array}\right)$ &
$(23)=\left(\begin{array}{ccc} 0 & 1 & 0 \\ 1 & 0 & 0 \\ 0 & 0 & 1\end{array}\right)$ &  $\left(\begin{array}{ccc} 1 & 0 & 0 \\ \pi o & o & 1 \\ \pi o & 1 & 0\end{array}\right)$ \\
\hline
$(12)=\left(\begin{array}{ccc} 0 & 1 & 0 \\ 1 & 0 & 0 \\ 0 & 0 & 1\end{array}\right)$ & $\left(\begin{array}{ccc} o & 1 & 0 \\ 1 & 0 & 0 \\ \pi o & \pi o & 1\end{array}\right)$ &
$(123)=\left(\begin{array}{ccc} 0 & 0 & 1 \\ 1 & 0 & 0 \\ 0 & 1 & 0\end{array}\right)$ & $\left(\begin{array}{ccc} o & o & 1 \\ 1 & 0 & 0 \\ \pi o & 1 & 0\end{array}\right)$ \\
\hline
$(132)=\left(\begin{array}{ccc} 0 & 1 & 0 \\ 0 & 0 & 1 \\ 1 & 0 & 0\end{array}\right)$ & $\left(\begin{array}{ccc} o & 1 & 0 \\ o & \pi o & 1 \\ 1 & 0 & 0\end{array}\right)$ &
$(13)=\left(\begin{array}{ccc} 0 & 0 & 1 \\ 0 & 1 & 0 \\ 1 & 0 & 0\end{array}\right)$ & $\left(\begin{array}{ccc} o & o & 1 \\ o & 1 & 0 \\ 1 & 0 & 0\end{array}\right)$ \\
\hline
\end{tabular}
\egroup
~\\

Let $N(k_F)$ be the $k_F$-points of $N$ (the upper triangular unipotent matrices with coefficients in $k_F$). $k_F$ has canonical (multiplicative) injection to \break $o_F\subset F$, hence any subgroup $H(k_F)\leq N(k_F)$ is mapped injectively to $N_0$ (however this is not a group homomorphism). We denote this subset of $N_0$ by $\widetilde{H(k_F)}$.

\begin{proposition}
\label{uxbdecomposition}
A set of double coset representatives of $U^{(1)}\setminus G/B$ is $\cup_{w\in W}\widetilde{N_w(k_F)}w$. Every element of $G$ can be written uniquely in the form $rb$ with $r\in R$ and $b\in B$.
\end{proposition}

\begin{proof}
By the Bruhat decomposition of $G(k_F)$ a set of double coset repre\-sentatives of
$U^{(1)}\setminus G_0/(B\cap G_0)$ is the set as above. Since $G=G_0B$, we have the first part of proposition.

Let $g=unwb\in G$ with $u\in U^{(1)}$, $w\in W$, $n\in\widetilde{N_w(k_F)}$ and $b\in B$. Then $g=w(w^{-1}nw)u'b$ with $u'=w^{-1}n^{-1}unw\in U^{(1)}$. But then there exist $n'\in N^-\cap U^{(1)}$ and $b'\in B$ such that $u'=n'b'$. Then $g=w(w^{-1}nwn')(b'b)$, where $w^{-1}nwn'\in r_w^{-1}(N_0(k_F))$ because of the definition of $N_w$.

For any $w\in W$ we clearly have $U^{(1)}\widetilde{N_w(k_F)}wB=R_wB$. Hence the uniqueness follows: if $rb=r'b'$ then there exists $w\in W$ such that $r,r'\in R_w$ and $b'b^{-1}=(r'^{-1}w^{-1})(wr)\in B\cap N^-=\{\mathrm{id}\}$.
\end{proof}

\begin{definition}
For any $w\in W$ let $U_w=U^{(1)}\widetilde{N_w(k_F)}wB$. This way we partitioned $G$ into open subsets indexed by the Weyl group. We obviously have $U_w=R_wB$.
\end{definition}

\begin{corollary}
For any $w\in W$ we have that $U_w$ is (left) $N_0$-invariant.
\end{corollary}

\begin{proof}
Let $n'\in N_0$ and $x=unwb\in U^{(1)}\widetilde{N_w(k_F)}wB$. We have \break $N_0=N_{0,w}(N'_w\cap N_0)$, thus $n'n=mm'$ for some $m\in N_{0,w}$ and $m'\in N'_w\cap N_0$, moreover we can write $m=m_1m_0\in(N_w\cap U^{(1)})\widetilde{N_w(k_F)}$. By the definition of $N'_w$
\[n'x=(n'un'^{-1}m_1)m_0w(w^{-1}m'wb)\in U^{(1)}\widetilde{N_w(k_F)}wB,\]
meaning that $U_w$ is $N_0$-invariant.
\end{proof}

\begin{proposition}
\label{rendezes}
Let $y\in U_w=R_wB$, $nt\in B_+=N_0T_+$, and \break $x=t^{-1}n^{-1}y\in U_{w'}=R_{w'}B$. Then $w'\preceq w$.
\end{proposition}

\begin{proof}
Let $y=rb$ with $r\in R_w$ and $b\in B$. By the previous proposition we may assume that $n=\mathrm{id}$. If $t=\mathrm{diag}(t_1,t_2,\dots,t_n)\in G_0$, then
\[x=w(w^{-1}t^{-1}w(w^{-1}r)w^{-1}tw)(w^{-1}t^{-1}wb),\]
where $w^{-1}t^{-1}w(w^{-1}r)w^{-1}tw\in r_w^{-1}(N_0(k_F))$, because it is in $N^-$ and the coefficients under the diagonal have the same valuation as those in $w^{-1}r$. $T_+$ as a monoid is generated by $T\cap G_0$, the center $Z(G)$ and the elements with the form $(\pi_F,\pi_F,\dots,\pi_F,1,1,\dots,1)$, hence it is enough to prove the proposition for such $t$-s.

So fix $t=(t_1=\pi_F,t_2=\pi_F,\dots,t_l=\pi_F,t_{l+1}=1,t_{l+2}=1,\dots,t_n=1)$, $r=(r_{ij})$ and try to write $x$ in the form as in Proposition \ref{uxbdecomposition}. For all $j=0,1,2,\dots,n$ we construct inductively a decomposition $x=(t^{(j)})^{-1}r^{(j)}b^{(j)}$ together with $w^{(j)}\in W$, where
\begin{itemize}
\item $w^{(j+1)}\preceq w^{(j)}$ for $j<n$ and such that the first $j$ columns of $w^{(j)}$ are the same as the first $j$ columns of $w^{(j+1)}$,
\item $t^{(j)}=\mathrm{diag}(t^{(j)}_i)\in T$ with
\[t^{(j)}_i=\left\{\begin{array}{ll} 1, & \mathrm{~if~} (w^{(j)})^{-1}(i)\leq j \\ t_i, & \mathrm{~if~} (w^{(j)})^{-1}(i)>j \end{array}\right.,\]
\item $r^{(j)}\in R_{w^{(j)}}$, and if we change the first $j$ columns of $r^{(j)}$ to the first $j$ columns of $(t^{(j)})^{-1}r^{(j)}$ it is still in $R_{w^{(j)}}$ (by de definition of $t^{(j)}$ it is enough to verify the condition for $(t^{(j)})^{-1}r^{(j)}$),
\item $b^{(j)}\in B$.
\end{itemize}
Then $w^{(n)}\preceq w^{(n-1)}\preceq w^{(n-2)}\preceq\dots\preceq w^{(1)}=w$. However for $j=n$ we have $t^{(n)}=\mathrm{id}$, hence $w^{(n)}=w'$ by disjointness of the sets $R_vB$ for $v\in W$, so we have the proposition.

For $j=0$ we have $t^{(0)}=t, r^{(0)}=r, b^{(0)}=b$ and $w^{(0)}=w$. From $j$ to $j+1$:
\begin{itemize}
\item If $w^{(j)}(j+1)\leq l$, then let $w^{(j+1)}=w^{(j)}$, so $t^{(j+1)}=e^{-1}_{w^{(j)}(j+1)}t^{(j)}$, where for $1\leq k\leq n$ we denote $e_k=e_k(\pi)$ the diagonal matrix with $\pi_F$ in the $k$-th row and 1 everywhere else. We can choose \break $r^{(j+1)}=e_{w^{(j)}(j+1)}^{-1}r^{(j)}e_{j+1}$, and $b^{(j+1)}=e_{j+1}^{-1}b^{(j)}$.

Then the first $j$ columns of $(t^{(j+1)})^{-1}r^{(j+1)}$ are equal of those of \break $(t^{(j)})^{-1}r^{(j)}$, and the entries at place $(i,j+1)$ with $i\neq w^{(j+1)}(j+1)$ are multiplied by $\pi_F$. Because of the conditions for $r^{(j)}$, this is in $R_{w^{(j+1)}}$. The other conditions for $w^{(j+1)},t^{(j+1)},r^{(j+1)}$ and $b^{(j+1)}$ obviously hold.
\item If $w^{(j)}(j+1)>l$ and if $\nu_F(r^{(j)}_{i,j+1})\geq 1$ for all $i\leq l$, then it suffices to choose $w^{(j+1)}=w^{(j)}, t^{(j+1)}=t^{(j)}, r^{(j+1)}=r^{(j)}$ and $b^{(j+1)}=b^{(j)}$.
\item Assume that $w^{(j)}(j+1)>l$ and that there exists $i\leq l$ such that $\nu_F(r^{(j)}_{i,j+1})=0$. Let $i_0$ be the maximal such $i$. Then choose \break $w^{(j+1)}(j+1)=i_0$, and $t^{(j+1)}=e_{i_0}^{-1}t^{(j)}$.

Let $r'=e_{i_0}^{-1}r^{(j)}e_{j+1}((r^{(j)}_{i_0,j+1})^{-1}\cdot\pi)$, where $e_j(\alpha)$ is the diagonal matrix with $\alpha\in F$ in the $j$-th row and 1 everywhere else. Note that $r'_{i_0,j+1}=1$ and $r'$ differs from $r^{(j)}$ only in the $i_0$-th row and the $j+1$-st column. But $(t^{(j+1)})^{-1}r'$ is not in $\mathrm{GL}_n(o_F)$ - for example $\nu_F(r'_{i_0,(w^{(j)})^{-1}(i_0)})=-1$, and there might be some other elements of $r'$ in the $i_0$-th row and columns between the $j+2$-nd and $j'=(w^{(j)})^{-1}(i_0)$-th.

To see this note first that $w^{(j)}(j+1)>l\geq i_0$, so $(w^{(j)})^{-1}(i_0)\neq j+1$. In particular the right multiplication with $e_{j+1}$ does not change the entry at place $(i_0,(w^{(j)})^{-1}(i_0))$. Since $r^{(j)}\in R_{w^{(j)}}$, the defining conditions of $R_{w^{(j)}}$ and that $(w^{(j)})^{-1}(i_0)\neq j+1$ imply $(w^{(j)})^{-1}(i_0)>j+1$. Thus $(t^{(j)}_{i_0})^{-1}=(t_{i_0})^{-1}=\pi_F^{-1}$, since $i_0\leq l$. By the definition of $R_{w^{(j)}}$ we have $r^{(j)}_{i_0,(w^{(j)})^{-1}(i_0)}=1$ . Therefore $r'_{i_0,(w^{(j)})^{-1}(i_0)}=\pi^{-1}$ which has valuation -1.

But note, that in the $j+1$-st column of $r'$ the $i_0$-th element is 1, all the other has valuation at least 1. Thus the first $j+1$ columns of $(t^{(j+1)})^{-1}r'$ satisfy the condition for the first $j+1$ columns of $(t^{(j+1)})^{-1}r^{(j+1)}$ - this is meaningful, because we already fixed the first $j+1$ columns of $w^{(j+1)}$.

So we want to find $r^{(j+1)}=r'b'$ with $b'\in B$ such that the first $j+1$ columns of $b'$ is those of the identity matrix, and \break $(t^{(j+1)})^{-1}r^{(j+1)}\in R_{w^{(j+1)}}$ with $w^{(j)}\preceq w^{(j+1)}$. 

Let $j_0=j+1$, and if $j_i<j'$ then
\[j_{i+1}=\min\{h|j+1<h, r'_{i_0,h}\notin o_F, w^{(j)}(j_i)>w^{(j)}(h)\}.\]
We claim that the set on the right hand side contains $j'$ if $j_i<j'$. We prove it by induction on $i$. For $i=0$ we already verified it. Assume by contradiction that $w^{(j)}(j_i)<i_0=w^{(j)}(j')$. Since $j'>j_i$ we get $r^{(j)}_{i_0,j_i}\in\pi_Fo_F$, because $r^{(j)}\in R_{w^{(j)}}$. But then $r'_{i_0,j_i}\in o_F$, because $r'\in e^{-1}_{i_0}r^{(j)}\cdot\mathrm{Mat}(o_F)$, contradicting the defining conditions of $j_i$. Thus we have $w^{(j)}(j_i)\geq i_0=w^{(j)}(j')$.

Let $s$ be minimal such that $j_s=j'$ and set $j_{s+1}=n+1$. We claim that $r^{(j+1)}$ will be in $R_{w^{(j+1)}}$ with $w^{(j+1)}=w^{(j)}(j_{s-1},j_s)(j_{s-2},j_{s-1})\dots(j_0,j_1)$. Then the condition $w^{(j+1)}\prec w^{(j)}$ holds, because the multiplication from right with each transposition $(j_i,j_{i+1})$ decreases the inversion number and the length respectively, by the definition of $j_{i+1}$.

For the existence of a $b'\in B$ such that $r'b'\in R_{w^{(j+1)}}$ we prove the following statements inductively:

\begin{lemma}
For all $j+1\leq k\leq n$ there exist
\begin{itemize}
\item $b'^{(k)}\in B$ such that the first $k$ column of $r'^{(k)}=r'b'^{(k)}$ satisfy the defining condition for the first $k$ column in $R_{w^{(j+1)}}$, and if we have $k<n$ then $r'^{(k)}$ and $r'^{(k+1)}$ differ only in the $k+1$-st column. 
\item a linear combination $s^{(k)}$ of the columns $j+1, j+2, \dots, k$ in $r'^{(k)}$ for which we have 
\[s^{(k)}_i=\left\{\begin{array}{ll}1, & \mathrm{~if~} i=i_0\\ 0, & \mathrm{~if~} (w^{(j+1)})^{-1}(i)\leq k, \mathrm{~and~} i\neq i_0 \\ \pi_Fx, & \mathrm{~for~some~} x\in o_F \mathrm{~otherwise} \end{array}\right.\]
and the maximal $i$ such that $\nu_F(s^{(k)}_i)=1$ is $w^{(j)}(j_{i'})$, where $i'$ is so, that $j_{i'}\leq k<j_{i'+1}$.
\end{itemize}
\end{lemma}

\begin{proof}
This holds for $k=j+1$ with $b'^{(j+1)}=\mathrm{id}$, $r'^{(j+1)}=r'$ and $s^{(j+1)}$ the $j+1$-st column of $r'$. To verify the condition for $s^{(j+1)}$ note that $r'_{(w^{(j)}(j+1),j+1)}=\pi$ and if $i>j+1$, then by the definition of $R_{w^{(j)}}$ we have that $r^{(j)}_{i,j+1}$ has valuation at least 1 and $r'_{(i,j+1)}=\pi_F(r^{(j)}_{i_0,j+1})^{-1}r^{(j)}_{i,j+1}$ has valuation at least 2.

Assume that we have $r'^{(k)}$, $b'^{(k)}$ and $s^{(k)}$. Let $i'$ be so that \break $j_{i'}\leq k<j_{i'+1}$ and $s'$ be the $k+1$-st column of $r'^{(k)}$ (which is equal with the $k+1$-st column of $r'$, thus for $i\neq i_0$ we have $s'_i=r^{(j)}_{i,k+1}$) and \break $s''=s'-r'^{(k)}_{(i_0,k+1)}s^{(k)}$. Then by the conditions on $s'$ we can change the $k+1$-st column of $r'^{(k)}$ to $s''$ with multiplication from right by an element $b''\in B$. Moreover $s''_{i_0}=0$, and the element in $s''$ with minimal valuation and biggest row index is the $w^{(j+1)}(k+1)$-st:
\begin{itemize}
\item If $\nu_F(r'^{(k)}_{(i_0,k+1)})\geq0$ then for $i\neq i_0$ we have $s'_i\equiv s''_i=s'_i-r'^{(k)}_{(i_0,k+1)}s^{(k)}_i\mod \pi_F$, hence the element with minimal valuation is in the row $w^{(j+1)}(k+1)=w^{(j)}(k+1)$ (because $r^{(j)}\in R_{w^{(j)}}$ and $j_{i'+1}\neq k+1$).
\item If $\nu_F(r'^{(k)}_{(i_0,k+1)})<0$ then it is -1 and for $i\neq i_0$ we have \break $s''_i=r^{(j)}_{(i,k+1)}-r'^{(k)}_{(i_0,k+1)}\cdot s^{(k)}_i$. Where on the right hand side the first term has positive valuation for $i>w^{(j)}(k+1)$ and 0 valuation for $i=w^{(j)}(k+1)$ (because $r^{(j)}\in R_{w^{(j)}}$), and the second has valuation 0=-1+1 for $i=w^{(j)}(j_{i'})$ and at least 1 for $i>w^{(j)}(j_{i'})$ (by the induction hypothesis on $s^{(k)}$). Moreover $j_{i'}\neq k+1$, because $j_{i'}\leq k$, hence $w^{(j)}(j_{i'})\neq w^{(j)}(k+1)$.

If $w^{(j)}(j_{i'})<w^{(j)}(k+1)$ then $j_{i'+1}\neq k+1$ and \break $w^{(j)}(k+1)=w^{(j+1)}(k+1)$. If $w^{(j)}(j_{i'})>w^{(j)}(k+1)$ then $j_{i'+1}=k+1$ and $w^{(j+1)}(k+1)=w^{(j+1)}(j_{i'+1})=w^{(j)}(j_{i'}).$

\end{itemize}
By multiplying this column with $(s''_{w^{(j+1)}(k+1)})^{-1}$ we get the element $r'^{(k+1)}$ (we also have to multiply the $k+1$-st row of $b''$ with $s''_{w^{(j+1)}(k+1)}$, this is $b'^{(k+1)}$). This satisfies the condition for the $k+1$-st row of $R_{w^{(j+1)}}$ because the defining conditions for $r^{(j)}\in R_{w^{(j)}}$, $s^{(k)}$ and the equality
\[\{i|(w^{(j+1)})^{-1}(i)<k+1\}=\{i|(w^{(j)})^{-1}(i)<k+1\}\setminus\{w^{(j)}(j_{i'})\}\cup\{i_0\}.\]

The last thing to verify is the existence of an appropriate linear combination $s^{(k+1)}$. Let $s^{(k+1)}=s^{(k)}-s^{(k)}_{w^{(j+1)}(k+1)}(s''_{w^{(j+1)}(k+1)})^{-1}\cdot s''.$ Since $\nu_F(s^{(k)}_{w^{(j+1)}(k+1)})>0$, we have $\nu_F(s^{(k+1)}_i)>0$ if $i\neq i_0$, and by the previous argument also $s^{(k+1)}_{w^{(j+1)}(j')}=0$ for $j'\leq k+1$ and $j'\neq j+1$.

If $w^{(j+1)}(k+1)>w^{(j)}(j_{i'})$, then $s^{(k)}_{w^{(j+1)}(k+1)}>1$ and $s^{(k+1)}\equiv s^{(k)}\mod \pi_F^2$. If $w^{(j+1)}(k+1)<w^{(j)}(j_{i'})$ then by the definition of $R_{w^{(j+1)}}$ for all $i>w^{(j+1)}(k+1)$ we have $\nu(s''_i)>1$ and again $s^{(k+1)}_i\equiv s^{(k)}_i\mod\pi_F^2$. If $w^{(j+1)}(k+1)=w^{(j)}(j_{i'})$, then by the definition of $R_{w^{(j)}}$ we have $s'_{w^{(j)}(j_{i'})}=r'_{(w^{(j)}(j_{i'}),k+1)}=0$, $s''_{w^{(j+1)}(k+1)}=0-r'^{(k)}_{(i_0,k+1)}s^{(k)}_{w^{(j)}(j_{i'})}$ and $s^{(k+1)}=$
\[=s^{(k)}-s^{(k)}_{w^{(j)}(j_{i'})}(-r'^{(k)}_{(i_0,k+1)}s^{(k)}_{w^{(j)}(j_{i'})})^{-1}\cdot\Big(s'-r'^{(k)}_{(i_0,k+1)}s^{(k)}\Big)=(r'^{(k)}_{(i_0,k+1)})^{-1}s',\]
which satisfies the condition because $s'$ is the $j_{i'+1}=k+1$-st column of $r'^{(k)}$ and because of the definition of $R_{w^{(j)}}$.
\end{proof}

To finish the proof we set $b'=b'^{(n)}$, $r^{(j+1)}=r'b'^{(n)}\in R_{w^{(j+1)}}$ and $b^{(j+1)}=(b'^{(n)})^{-1}(r^{(j)}_{i_0,j+1}\cdot e_{j+1}^{-1})b^{(j)}\in B$.
\end{itemize}
\end{proof}

\begin{corollary}
For any $w\in W$ we have $BwB=N_wwB\subset\cup_{w'\preceq w}U_{w'}$. In particular for any $0<m_0\leq n!$ we have that 
\[\bigcup_{m\geq m_0}U_{w_m}\subset G\setminus G_{m_0-1}=\bigcup_{m\geq m_0}Bw_mB.\]
\end{corollary}

\begin{proof}
Let $x=n_wwb\in N_{w}wB$. Then there exists $t\in T_+$ such that \break $n'=tn_wt^{-1}\in N_0$. Thus $x=t^{-1}n'w(w^{-1}tw)b=t^{-1}n'wb''$ with $b''\in B$. By the previous proposition for $w=w\cdot\mathrm{id}\in R_wB$ and $(n')^{-1}t\in B_+$, there exist $w'\prec w$, $r_{w'}\in R_{w'}$ and $b'\in B$ such that $t^{-1}n'w=r_{w'}b'$, hence $x=r_{w'}(b'b'')\in U_{w'}$. The second assertion follows from that:
\[\bigcup_{m\geq m_0}U_{w_m}=G\setminus\bigcup_{1\leq m<m_0}U_{w_m}\subset G\setminus\bigcup_{1\leq m<m_0}Bw_mB=G\setminus G_{m_0-1}.\]
\end{proof}

\begin{remark}
We can achieve the results of this section not only for $\mathrm{GL}_n$, but different groups: let $G'$ be such that
\begin{itemize}
\item $G'$ is isomorphic to a closed subgroup in $G$ which we also denote by $G'$,
\item In $G'$ a maximal torus is $T'=T\cap G'$, a Borel subgroup $B'=B\cap G'$ with unipotent radical $N'=N\cap G'$, such that $N_{G'}(T')=N_G(T)\cap G'$ and hence $W'\leq W$ with $w_0\in W'$, with representatives $w'$ of $W'$ in $G'_0\leq G_0$ such that the representatives $w$ of $W$ in $G$ can be written in the form $w=w't$ such that $t\in T\cap G_0$.
\item $G'_0=G_0\cap G'$ with $G'=G'_0B'$ and
\item $U'^{(1)}=U^{(1)}\cap G'$ such that $U'^{(1)}\subset (N'^-\cap U'^{(1)})B'$ for $N'^-=w_0N'w_0$.
\end{itemize}

For example these condititons are satisfied for the group $\mathrm{SL}_n$.

The proof of the first proposition works for such $G'$, and from a decomposition $x=r'b'\in R_w'B'\subset G'$ we get some $r\in R_w$ and $b\in B$ such that $x=rb\in G$. Hence the $B_+'$-action on $G'$ respects the restriction of $\prec$ to $W'$ in the sense that if $x\in R_{w'}B'$ and $b'\in B'$ then there exists $w''\preceq w'$ in $W'$ such that $b'^{-1}x\in R'_{w''}B'$.
\end{remark}

\section{Generating $B_+$-subrepresentations}

For any torsion $o_K$-module $X$ with $o_K$-linear $B$-action denote the (partially ordered) set of generating $B_+$-subrepresentations of $X$ (those $B_+$-submodules $M$ of $X$ for which $BM=X$) by $\mathcal{B}_+(X)$.

For example $\mathrm{Ind}^{U_{w_0}}_B(\chi)\simeq C^\infty(N_0)$ is the minimal generating $B_+$-sub\-rep\-resen\-tation of the Steinberg representation $V_{n!-1}=\mathrm{Ind}^{Bw_0B}_B(\chi)\simeq C^\infty_c(N)$. (cf \cite{SchVi}, Lemma 2.6)

\begin{proposition}
Let $X$ be a smooth admissible and irreducible torsion $o_K$-representation of $G$. Then $M_0=B_+X^{U^{(1)}}$ is a generating $B_+$-subrepresen\-tation of $X$. For any $M\in\mathcal{B}_+(X)$ there exists a $t_+\in T_+$ such that \break $t_+M_0\subset M$.
\end{proposition}

\begin{proof}
$X$ is a $\pi_K$ vectorspace as well, because $\pi_KX\leq X$, hence by the irreducibility it is either $0$ or $X$, and since $X$ is torsion $\pi_KX=X$ gives $X=0$.

$BM_0$ is a $B$-subrepresentation, and also a $G_0$-subrepresentation (because $U^{(1)}\lhd G_0$). $G_0B=BG_0=G$, so $BM_0$ is a $G$-subrepresentation of $X$. $M_0$ is not $\{0\}$, since $U^{(1)}$ is pro-$p$ and since $X$ is irreducible $BM_0=X$, hence $M_0$ is generating. And $M_0$ is clearly a $B_+$-submodule of $X$.

$X$ is admissible, hence $X^{U^{(1)}}$ has a finite generating set, say $R$. Let $M$ be as in the proposition. For any $r\in R$ there exists an element $t_r\in T_+$ such that $t_rr\in M$ (\cite{SchVi}, Lemma 2.1). The cardinality of $R$ is finite, hence for $t_+=\prod_{r\in R}t_r$ we have $t_r^{-1}t_+\in T_+$ for all $r\in R$, and then $t_+M_0\subset M$.
\end{proof}

From now on let $V=\mathrm{Ind}^{G}_{B}(\chi)$ as before and $M_0=B_+V^{U^{(1)}}$. Then $V^{U^{(1)}}$ (as a vector space) is generated by
\[f_r:\left\{\begin{array}{ccc} urb & \mapsto & \chi^{-1}(b) \\ y\neq urb & \mapsto & 0 \end{array}\right. \qquad\Bigg(r\in U^{(1)}\setminus G/B=\bigcup_{w\in W}\widetilde{N_w(k_F)}w\Bigg).\]

If we denote the coset $U^{(1)}wB$ also with $w$, then $V^{U^{(1)}}$ is generated by $\{f_w|w\in W\}$ as an $N_0$-module. Hence any $f\in M_0$ can be written in the form $\sum_{i=1}^s\lambda_in_it_if_{w_i}$ for some $\lambda_i\in k_K, n_i\in N_0, t_i\in T_+$ and $w_i\in W$.

\begin{proposition}\label{mingen}
$M_0$ is minimal in $\mathcal{B}_+(V)$.
\end{proposition}

\begin{remark}
In \cite{SchVi} section 12 Schneider and Vigneras treated the case of the subquotients $V_{m-1}/V_m$. Unfortunately $M_0$ does not generally give the minimal generating $B_+$-subrepresentation of $V_{m-1}/V_m$ on this subqoutient, since that their method does not work on the whole $V$. It is not true even for $\mathrm{GL}_3(\mathbb{Q}_p)$: an explicit example is shown in Corollary \ref{M0tulnagy}.
\end{remark}

\begin{proof}
By the previous proposition, it is enough to show, that for any $t'\in T_+$ we have $M_0\subset B_+t'M_0$.

If $t'\in G_0$, then $t'^{-1}\in T_+$ thus we have $B_+t'=B_+$, and \break $B_+t'M_0=B_+M_0=M_0$. The same is true for central elements $t'\in Z(G)$. So it is enough to prove for $t'=(\pi_F,\pi_F,\dots,\pi_F,1,1,\dots,1)$ that \break $M_0\subset B_+t'M_0$.

Let $j_0\in\mathbb{N}$ be such that $t'_{j_0}=\pi_F$ and $t'_{j_0+1}=1$. We need to show, that for all $w\in W$ we have $f_w\in B_+t'M_0$. We prove it by descending induction on $w$ with respect to $\prec$.

Let us denote $N^{(1)}_{j_0}=\{n\in N\cap U^{(1)}|\forall i<j, (j_0-i)(j-j_0)<0: n_{ij}=0\}$, $N_{w,j_0}=N_w\cap N^{(1)}_{j_0}$ and
\[\Theta_{w,j_0}=\{\mathrm{a~set~of~representatives~of~}N_{w,j_0}/t'N_{w,j_0}t'^{-1}\}\subset N_0\cap U^{(1)}.\]

It is enough to prove the following:

\begin{lemma}
\label{mingenlemma}
Let $g=\sum_{m\in\Theta_{w,j_0}}mt'f_w$. Then $\chi(w^{-1}t'w)f_w-g$ is in \break $\sum_{w':w\prec w'}N_0f_{w'}$.
\end{lemma}

We claim that for $r\in R_w$ we have
\[t'f_w(r)=\left\{\begin{array}{ll}\chi(w^{-1}t'w), &\mathrm{if~} \forall i\leq j_0<j, w^{-1}(i)>w^{-1}(j): r_{ij}\in \pi_F^2o_F,\\ 0, & \mathrm{otherwise}. \end{array}\right.\] 

$t'f_w(r)=f(t'^{-1}r)$ is nonzero if and only if $t'^{-1}r\in U^{(1)}wB$. Following the proof of Proposition \ref{rendezes}, it is equivalent to that for all $1\leq j\leq n$ we have $w=w^{(j)}$ and that the first $j$ column of $(t^{(j)})^{-1}r^{(j)}$ is as the first $j$ column of $U^{(1)}w$. This holds if and only if $r_{ij}\in\pi_F^2o_F$ for all $i$ and $j$ as above. Then we have $r^{(n)}=t'^{-1}rw^{-1}t'w$ and $b^{(n)}=w^{-1}(t')^{-1}w$, hence our claim.

Therefore $\chi(w^{-1}t'w)f_w|_{U_w}=\sum_{m\in\Theta_{w,j_0}}mt'f_w|_{U_w}$. Hence by the induction hypothesis and Proposition \ref{rendezes} it suffices to prove that $g$ is $U^{(1)}$-invariant.

To do that, first notice that since $f_w$ is $U^{(1)}$-invariant, we have that $t'f_w$ is $t'U^{(1)}t'^{-1}$-invariant. Moreover, since for all $m\in\Theta_{w,j_0}$ we have \break $m\in N_0\cap U^{(1)}\subseteq t'N_0t'^{-1}$, $m$ normalizes $t'U^{(1)}t'^{-1}$, $mt'f_w$ is also $t'U^{(1)}t'^{-1}$-invariant, and so is $g$.

On the other hand, we can write
\[g=\sum_{m\in\Theta_{w,j_0}}mt'f_w=\sum_{m\in\Theta_{w,j_0}}t'(t'^{-1}mt')f_w=t'\Bigg(\sum_{n\in t'^{-1}N_{w,j_0}t'/N_{w,j_0}}nf_w\Bigg),\]
where the sum in the bracket on the right hand side is obviously $t'^{-1}N_{w,j_0}t'$-invariant, hence $g$ is $N_{w,j_0}$-invariant.

Denote $N'_{w,j_0}=N'_w\cap N^{(1)}_{j_0}$. Then $N_{w,j_0}$ centralizes $t'^{-1}N'_{w,j_0}t'$: let \break $n_0=\mathrm{id}+m_0\in t'^{-1}N'_{w,j_0}t'$, $n\in N_{w,j_0}$,
\[(n^{-1}n_0n-n_0)_{xy}=(n^{-1}m_0n-m_0)_{xy}=\sum_{x\leq s\leq t\leq y}(n^{-1})_{xs}(m_0)_{st}n_{ty}-(m_0)_{xy},\]
and by the definition $N^{(1)}_{j_0}$, $(m_0)_{st}$ is 0, unless $s\leq j_0\leq t$ and hence \break $(n^{-1})_{xs}m_{st}n_{ty}=0$, unless $x=s$ and $y=t$.

By the definiton of $N'_w$ we have $w^{-1}N'_{w,j_0}w\subset B$, so for any $u\in U^{(1)}$ and $n_0\in t'^{-1}N'_{w,j_0}t'\subset G_0$ we have $n_0uw=(n_0un_0^{-1})w(w^{-1}n_0w)\in U^{(1)}wB$, and hence $f_w$ is $t'^{-1}N'_{w,j_0}t'$-invariant.

Altogether for any representative $n\in\Theta_{w,j_0}$
\[nf_w(n_0x)=f_w(n^{-1}n_0x)=f_w(n_0n^{-1}x)=f_w(n^{-1}x)=nf_w(x),\]
meaning that $nf_w$ is $t'^{-1}N'_{w,j_0}t'$-invariant, and $t'nf_w$ is $N'_{w,j_0}$-invariant. So $g$ is also $N'_{w,j_0}$-invariant. 

$U^{(1)}$ is contained in $\big<t'U^{(1)}t'^{-1}, N_{w,j_0}, N'_{w,j_0}\big>$, so $g$ is $U^{(1)}$-invariant, and we are done.
\end{proof}

\begin{corollary}
For any $f\in M_0$ there exists $t\in T_+$ such that $f$ can be written in form $\sum_{i=1}^s\lambda_in_itf_{w_i}$ for some $\lambda_i\in k_K, n_i\in N_0$ and $w_i\in W$.
\end{corollary}

~\\

Define the $k_K[B_+]$-submodules $M_{0,m}=\sum_{m'>m} B_+f_{w_{m'}}\leq\mathrm{Ind}^{G_m}_{B}(\chi)$. We obtain a descending filtration $M_0=M_{0,0}\geq M_{0,1}\geq\dots\geq M_{0,n!}=0$. Then $M_{0,n!-1}=\mathrm{Ind}^{U_{w_0}}_B(\chi)$ is the minimal generating subrepresentation of $V_{n!-1}$.

\begin{proposition}
Let $1<m\leq n!$, $w=w_{m-1}$ and $n'\in N'_{0,w}=N'_w\cap N_0$ and $t\in T_+$. Then $g=n'tf_w-nf_w\in M_{0,m}$.
\end{proposition}

\begin{proof}
For $w'\prec w$ we have $tf_w|_{U_{w'}}=n'tf_w|_{U_{w'}}=0$ and following the proof of Proposition \ref{rendezes} we get $n'tf_w|_{U_w}=tf_w|_{U_w}$. Moreover $g$ is $tU^{(1)}t^{-1}$-invariant, thus it is contained in $\sum_{m'>m-1}tf_{w_{m'}}\subset M_{0,m}$.
\end{proof}

\begin{corollary}
For any $f\in M_0$ there exists $t\in T_+$ such that $f$ can be written in form $\sum_{i=1}^s\lambda_in_itf_{w_i}$ for some $\lambda_i\in k_K$, $w_i\in W$ and $n_i\in N_{0,w_i}$.
\end{corollary}

\begin{remarks}
\begin{enumerate}
\item $V$ is the modulo $\pi_K$ reduction of the $p$-adic principal series representation. This can be done with any $l\in\mathbb{N}$ for the modulo $\pi_K^l$ reduction. Then the $\pi_K$-torsion part of the minimal generating $B_+$-representation is exactly $M_0$.
\item This can be carried out in the same way for groups $G'$ as in the previous section satisfying moreover $N_0\subset G'$. For example $G'=\mathrm{SL}_n$ has this property (but its center is not connected), or $G'=P$ for arbitary $P\leq G$ parabolic subgroup has also (but these are not reduvtive).
\end{enumerate}
\end{remarks}

\section{The Schneider-Vigneras functor}

Following Schneider and Vigneras (\cite{SchVi}, section 2) we introduce the functor $D$ from torsion $o_K$-modules to modules over the Iwasawa algebra of $N_0$.

Let us denote the completed group ring of $N_0$ over $o_K$ by $\Lambda(N_0)$, and define
\[D(X)=\lim_{\overrightarrow{M\in\mathcal{B}_+(X)}}M^*,\]
as an $\Lambda(N_0)$-module, equipped with the natural $T_+^{-1}$-action $\psi$.

On $D(V)$ the action of $\pi_K$ is 0, hence we can view it as a \break $\Omega(N_0)=\Lambda(N_0)/\pi_K\Lambda(N_0)$-module.

By Proposition \ref{mingen} we have

\begin{proposition}
The $\Omega(N_0)$-module $D(V)$ is equal to $M_0^*$.
\end{proposition}

\begin{remarks}
\begin{enumerate}
\item We do not now whether $D(V)$ is finitely generated or it has rank 1 as an $\Omega(N_0)$-module.
\item On $M_0$ we have an action of $U^{(1)}$: if $x\in U^{(1)}$, $n\in N_0,t\in T_+$ and $w\in W$ then we can write $n^{-1}xn=n_1n_2\in U^{(1)}$ with $n_1\in N_0$ and $n_2\in B^-T\cap U^{(1)}$ (with $B^-=N^-T$), thus
\[xntf_w=n(n^{-1}xn)tf_w=(nn_1)t(t^{-1}n_2t)f_w=(nn_1)tf_w\in M_0,\]
since $t^{-1}n_2t\in U^{(1)}$ and $f_w$ is $U^{(1)}$-invariant. Thus on $D(V)$ there is an action of $\Lambda(U^{(1)})$, therefore an action of $\Lambda(I)$ (with $I$ denoting the Iwahori subgroup).
\end{enumerate}
\end{remarks}

Till this point we considered only the $\Lambda(N_0)$-module structure of $D(V)$. Now we shall examine the $\psi$-action as well. We need to get an \'etale module from $D(V)$, thus we examine the $\psi$-invariant images of $D(V)$ in an \'etale module.

Let $D$ be a topologically \'etale (see \cite{SVZ} the first lines of Section 4) $(\varphi,\Gamma)$-module over $\Omega(N_0)$, with the following properties:
\begin{itemize}
\item $D$ is torsion-free as an $\Omega(N_0)$-module,
\item on $D$ the topology is Hausdorff,
\item $D$ has a basis of neighborhoods of 0, containing $\varphi$-invariant $\Omega(N_0)$-submodules ($O\leq D$ open such that $\varphi_t(O)\subseteq O$ for all $t\in T_+$).
\end{itemize}

\begin{theorem}\label{mainthm}
Let $D$ be as above and $F:D(V)\to D$ a continuous $\psi$-invariant map (where $\psi$ is the canonical left inverse of $\varphi$ on $D$). Then $F$ factors through the natural map $F_0:D(V)\to D(V_{n!-1})$: there exists a continuous $\psi$-invariant map $G:D(V_{n!-1})\to D$ such that $F=F_0\circ G$.
\end{theorem}

\begin{proof}
$\overline{D(V)-tors}$ is in the kernel of $F$ (the torsion submodules exist, because the rings are Ore rings).

In $M_0/(M_0\cap V_{n!-1})$ there are no nontrivial $k_K[N_0]$-divisible elements, because if $f\in M_0$ the image of it in $M_0/(M_0\cap V_{n!-1})$ is $f'=f|_{G\setminus Bw_0B}$. Assume by contradiction that $f'$ is $k_K[N_0]$-divisible. If it is nontrivial, then there exists $bw_mb\in G$ such that $f(bw_mb)\neq0$ with some $m<n!$ Let \break $n'\in N'_{0,w_m}=N_0\cap w_mN_0w_m^{-1}$ with $n'\neq\mathrm{id}$, and $[n']-[\mathrm{id}]\in k_K[N_0]$. Then for any $g\in M_0$ we have
\[([n']-[\mathrm{id}])g(w_m)=g(n'^{-1}w_m)-g(w_m)=g(w_m(w_m^{-1}n'^{-1}w_m))-g(w_m)=0,\]
because $w_m^{-1}n'^{-1}w_m\in N$. Thus $f'$ is not divisible by $[n']-[\mathrm{id}]$.

It follows that $F$ factors through $(M_0\cap V_{n!-1})^*$: The fact that there are no nontrivial divisible submodules in $M_0/(M_0\cap V_{n!-1})$  implies that for any (closed) submodule the maps $f\mapsto\lambda f$ are not surjective for all \break $\lambda\in k_K[N_0]^*$. Hence dual maps are not injective for all $\lambda$ - the dual has no torsionfree quotient arising as a dual of a submodule of $M_0/(M_0\cap V_{n!-1})$, thus $(M_0/(M_0\cap V_{n!-1}))^*\leq\overline{D(V)-tors}$. Now consider the exact sequence
\[0\to M_0\cap V_{n!-1}\to M_0\to M_0/(M_0\cap V_{n!-1})\to 0.\]

We claim that $F$ factors through $M_{0,n!-1}^*$ as well. If $f\in(M_0\cap V_{n!-1})^*$ such that $f|_{M_{0,n!-1}}\equiv 0$, then $\psi_t(u^{-1}f)|_{t^{-1}M_{0,n!-1}}\equiv 0$ for all $t\in T_+$ and $u\in N_0$:

The $\psi$-action on $D(V)$ comes from the $T_+$-action on $V$, hence \break $\psi_t(u^{-1}f)(t^{-1}x)=(u^{-1}f)(tt^{-1}x)=f(ux)=0$ if $x\in M_{0,n!-1}$.

For all $O\subseteq D$ open subset there exists $t\in T_+$ such that \break $\mathrm{Ker}(f\mapsto f|_{t^{-1}M_{0,n!-1}})\subset F^{-1}(O)$, since $F$ is continuous and \break $\bigcup_{t\in T_+}t^{-1}M_{0,n!-1}=V_{0,n!-1}$. If $O$ is $\varphi$ and $N_0$-invariant as well, then
\[F(f)=\sum_{u\in N_0/tN_0t^{-1}}u\varphi_t(F(\psi_t(u^{-1}f))\subseteq O.\]
Then $F(f)=0$ by the Hausdorff property.

By \cite{SchVi}, Proposition 12.1, we have $D(V_{n!-1})=M_{0,n!-1}^*$, which completes the proof. 
\end{proof}

\begin{remarks}
\begin{enumerate}
\item For this we do not need the $\Gamma$-action of $D$, the statement is true for $D$ \'etale $\varphi$-modules with continuous $N_0$ and $\varphi$-action.
\item Let $D'$ be the maximal quotient of $D(V)$, which is torsionfree, Haussdorff and on which the action of $\psi$ is nondegenerate in the following sense: for all $d\in D'\setminus\{0\}$ and $t\in T_+$ there exists $u\in N_0$ such that $\psi_t(ud)\neq0$. Then the natural map from $D'$ to $D(V_{n!-1})$ is bijective.
\item By \cite{Za} section 4 if $F=\mathbb{Q}_p$, we have that $D^0(V_{n!-1})=D(V_{n!-1})$ and $D^i(V_{n!-1})=0$ for $i>0$.
\end{enumerate}
\end{remarks}

Following \cite{SchVi} we choose a surjective homomorphism \break $\ell:N_0\to\mathbb{Q}_p$. Then we can get $(\varphi,\Gamma)$-modules from $D(V)$: Let $\Lambda_\ell(N_0)$ denote the ring $\Lambda_{N_1}(N_0)$ of \cite{SchVi} with $N_1=\mathrm{Ker}(\ell)$, with maximal ideal $\mathcal{M}_\ell(N_0)$, $\Omega_\ell(N_0)=\Lambda_\ell(N_0)/\pi_K\Lambda_\ell(N_0)$. The ring $\Lambda(N_0)$ can be viewed as the ring $\Lambda(N_1)[[X]]$ of skew Taylor series over $\Lambda(N_1)$ in the variable $X = [u]-1$ where $u\in N_0$ and $(u)$ is a topological generator of $\ell(N_0)=\mathbb{Z}_p$. Then $\Lambda_\ell(N_0)$ is viewed as the ring of infinite skew Laurent series $n\in\mathbb{Z}a_nX^n$ over $\Lambda(N_1)$ in the variable $X$ with $\lim_{n\to-\infty}a_n=0$ for the compact topology of $\Lambda(N_1)$.

Let $D_\ell(V)=\Lambda_\ell(N_0)\otimes_{\Lambda(N_0)}D(V)$.

\begin{corollary}
Let $D$ be a finitely generated topologically \'etale $(\varphi,\Gamma)$-module over $\Omega_\ell(N_0)$, and $F':D_\ell(V)\to D$ a continuous map. Then $F'$ factors through the natural map $F'_0:D_\ell(V)\to D_\ell(V_{n!-1})$.
\end{corollary}

\begin{proof}
If $D$ is a finitely generated topologically \'etale $(\varphi,\Gamma)$-module over $\Omega_\ell(N_0)$, then it automatically satisfies the conditions above:

$D$ is \'etale, hence $\Omega_\ell(N_0)$-torsion free (Theorem 8.20 in \cite{SVZ}), thus $\Omega(N_0)$-torsion free as well. It is Hausdorff, since finitely generated and the weak topology is Haussdorff on $\Omega_\ell(N_0)$ (Lemma 8.2.iii in \cite{SchVi}).

Finally we need to verify the condition for the neighborhoods. The sets $\mathcal{M}_\ell(N_0)^kD+\Omega(N_0)\otimes_{k[[X]]}X^n\ell(D)^{++}$ (where $\ell(D)$ is the \'etale $(\varphi,\Gamma)$-module attached to $D$ at the category equivalence \cite{SVZ} Theorem 8.20) are open $\varphi$-invariant $\Omega(N_0)$ submodules and form a basis of neighborhoods of 0 in the weak topology of $D$.

Thus $D(V)\to D_\ell(V)\to D$ factors through $D(V)\to D(V_{n!-1})$, hence the corollary.
\end{proof}

\section{Some properties of $M_0$}

In this section we point out some properties of $M_0$, which make the picture more difficult than the known case of subqoutients $V_{m-1}/V_m$. Recall (\cite{SchVi} section 12) that $V_{m-1}/V_m\simeq V(w_m,\chi)$, which has a minimal generating $B_+$-subrepresentation
\[M(w_m,\chi)=C^\infty(N_0/N'_{w_m}\cap N_0)\in\mathcal{B}_+(V(w_m,\chi)).\]

\begin{proposition}
Let $n=3$, $F=\mathbb{Q}_p$, then $M_0\cap V_{n!-1}\supsetneq M_{0,n!-1}$.
\end{proposition}

\begin{corollary}\label{M0tulnagy}
Thus $M_0\cap V_{n!-1}$ is not equal to the minimal generating $B_+$-subrepresentation of $V_{n!-1}$, which is $C^\infty(N_0)=M_{0,n!-1}$ (\cite{SchVi} section 12).
\end{corollary}

\begin{proof}
Assume that $\chi=\chi_1\otimes\chi_2\otimes\chi_3:T\to k_K^*$ is a character, such that neither $\chi_1/\chi_2$, nor $\chi_2/\chi_3$ is trivial on $o_K^*$. Similar construction can be carried out in the other cases.

Let $\prec_T$ be the following total ordering of the Weyl group of $G$ refining the Bruhat ordering:
\[w_1=\left(\begin{array}{ccc}1&0&0\\0&1&0\\0&0&1\end{array}\right)\prec_T w_2=\left(\begin{array}{ccc}0&1&0\\1&0&0\\0&0&1\end{array}\right)\prec_T w_3=\left(\begin{array}{ccc}1&0&0\\0&0&1\\0&1&0\end{array}\right)\prec_T\]
\[\prec_T w_4=\left(\begin{array}{ccc}0&1&0\\0&0&1\\1&0&0\end{array}\right)\prec_T w_5=\left(\begin{array}{ccc}0&0&1\\1&0&0\\0&1&0\end{array}\right)\prec_T w_6=\left(\begin{array}{ccc}0&0&1\\0&1&0\\1&0&0\end{array}\right)=w_0.\]

And let
\[h=\sum_{a=0}^{p^{2}-1}\sum_{b=0}^{p^2-1}\left(\begin{array}{ccc}1&a&b\\0&1&0\\0&0&1\end{array}\right)\left(\begin{array}{ccc}p^2&0&0\\0&1&0\\0&0&1\end{array}\right)f_{w_2}\in M_0,\]
\[f=h-\frac{1}{\chi_3(p^2)}\sum_{a=0}^{p^3-1}\sum_{b=0}^{p^3-1}h\Bigg(\left(\begin{array}{ccc} a&b&1 \\ 1&0&0 \\ 0&1&0\end{array}\right)\Bigg)\left(\begin{array}{ccc}1&a&b\\0&1&0\\0&0&1\end{array}\right)f_{w_5}.\]

Then it is easy to verify that $f\in M_0\cap V_{5}$, and that $f(z)\neq 0$ for
\[z=\left(\begin{array}{ccc}p^2&0&1\\1&0&0\\p&1&0\end{array}\right)\in Bw_0B\setminus N_0w_0B.\]
Thus $f\notin M_{0,5}=B_+f_6\subseteq\{f\in V|\mathrm{supp}(f)\leq N_0w_0B\}$.
\end{proof}

However, if $f\in M_0\cap V_{5}$ then $\mathrm{supp}(f)$ is contained in $Bw_0B\cap\bigcup_{i>3}R_iB$: A straightforward computation shows that for any $n\in N_0$, $t\in T_+$, $w\in W$ and
\begin{itemize}
\item for any $r\in R_{w_1}$ we have $ntf_w(r)=ntf_w(w_1)$. Let $r'=w_1\in G_5$,
\item for any $r\in R_{w_2}$ we have $ntf_w(r)=ntf_w(r')$ for
\[r'=\left(\begin{array}{ccc} \alpha & 1 & 0 \\ 1 & 0 & 0 \\ \beta' & 0 & 1 \end{array}\right)\in G_5, \mathrm{~where~} r=\left(\begin{array}{ccc} \alpha & 1 & 0 \\ 1 & 0 & 0 \\ \beta' & \gamma' & 1 \end{array}\right),\]
\item for any $r\in R_{w_3}$ we have $ntf_w(r)=ntf_w(r')$ for
\[r'=\left(\begin{array}{ccc} 1 & 0 & 0 \\ \alpha'-\beta\gamma' & \gamma & 1 \\ 0 & 1 & 0 \end{array}\right)\in G_5, \mathrm{~where~} r=\left(\begin{array}{ccc} 1 & 0 & 0 \\ \alpha' & \gamma & 1 \\ \beta' & 1 & 0 \end{array}\right).\]
\end{itemize}
Thus if $i<4$ and $r\in R_{w_i}$, then since $r'\notin Bw_0B$ we have $f(r)=f(r')=0$. 

\begin{proposition}
The quotients $M_{0,m-1}/M_{0,m-1}\cap V_m$ via $f\mapsto f(\cdot w_m)$ are isomorphic to $M(w_m,\chi)$.
\end{proposition}

\begin{proof}
It is obvious, that $f(\cdot w_m)\equiv0$ implies $f|_{G_m\setminus G_{m-1}}\equiv 0$ and \break $f\in M_{0,m-1}\cap V_m$. Hence the map $M_{0,m-1}/M_{0,m-1}\cap V_m\to M(w_m,\chi)$, \break $f\mapsto f(\cdot w_m)$ is injective.

Let $t_0=\mathrm{diag}(\pi_F^{n-1},\pi_F^{n-2},\dots,\pi_F,1)\in T_+$, and for any $l\in\mathbb{N}$ let \break $U^{(l)}=\mathrm{Ker}(G_0\to G(o_F/\pi_F^lo_F))$. For $x=rb\in R_{w_m}B$ we have
\[\sum_{n\in (N_0\cap U^{(l)})/t_0^lN_0t_0^{-l}}nt_0^lf_{w_m}(rb)=\left\{\begin{array}{ll} \chi^{-1}(b), & \mathrm{~if~}r\in U^{(l)}w_m,\\ 0, & \mathrm{~if~not.}\end{array}\right.\]
The image of these generate $M(w_m,\chi)$ as an $N_0$-module, so $f\mapsto f(\cdot w_m)$ is surjective.
\end{proof}

Since $M_{0,m}\leq V_m$, $M(w_m,\chi)$ is naturally a quotient of $M_{0,m-1}/M_{0,m}$, we have $D(V_{m-1}/V_m)\leq(M_{0,m-1}/M_{0,m})^*$.

\begin{proposition}
For $m=1$ and $m=n!-n+1,n!-n+2,\dots,n!$ \break $(M_{0,m-1}/M_{0,m})^*=D(V_{m-1}/V_m)$. For other $m$-s it is not true, for example if $n=3$, $F=\mathbb{Q}_p$ and $m=2,3$.
\end{proposition}

\begin{proof}
By the previous proposition it is enough to show that \break $M_{0,m}=M_{0,m-1}\cap V_m$ for $m=1$ and $m>n!-n$.

For $m=1$ the quotient is obviously $k_K$, for $m>n!-n$ we have \break $w\prec w_m$ implies $w=w_{n!}$, so if $f\in B_+f_{w_m}\cap V_{m-1}=B_+f_{w_m}\cap V_{n!-1}$, then \break $\mathrm{supp}(f)\subset U^{(1)}R_{w_{n!-1}}^{(1)}B$. But
\[M_{0,n!-1}\simeq C^\infty(N_0)\simeq\{f\in V_{n!-1}|\mathrm{supp}(f)\subset U^{(1)}R_{w_{n!-1}}B\}.\]

The fuction $f$ constructed in the beginning of this section is in \break $M_{0,1}\cap V_2\setminus M_{0,2}$. The same can be done for $m=3$.
\end{proof}

\end {document}